\documentclass[review]{elsarticle}
\input{mathrsfs.sty}
\usepackage{enumerate}
\usepackage{amsmath, amsthm, amscd, amsfonts, amssymb}
\newtheorem{theorem}{Theorem}[section]
\newtheorem{lemma}[theorem]{Lemma}
\newtheorem{proposition}[theorem]{Proposition}

\theoremstyle{definition}
\newtheorem{definition}[theorem]{Definition}

\numberwithin{equation}{section}
\usepackage{lineno,hyperref}
\makeatletter
\def\ps@pprintTitle{%
   \let\@oddhead\@empty
   \let\@evenhead\@empty
   \let\@oddfoot\@empty
   \let\@evenfoot\@oddfoot
}
\makeatother
\begin{document}

\begin{frontmatter}

\title{Almost Periodicity and Ergodic Theorems for Nonexpansive Mappings and Semigroups in Hadamard Spaces}


\author[mymainaddress]{Hadi Khatibzadeh\corref{mycorrespondingauthor}}
\cortext[mycorrespondingauthor]{Corresponding author}
\ead{hkhatibzadeh@znu.ac.ir}

\author[mymainaddress]{Hadi Pouladi}
\ead{hadi.pouladi@znu.ac.ir}

\address[mymainaddress]{Department of Mathematics, University of Zanjan, P. O. Box 45195-313, Zanjan, Iran.}

\begin{abstract}
The main purpose of this paper is to prove the mean ergodic theorem for nonexpansive mappings and semigroups in locally compact Hadamard spaces, including finite dimensional Hadamard manifolds. The main tool for proving ergodic convergence is the almost periodicity of orbits of a nonexpansive mapping. Therefore, in the first part of the paper, we study almost periodicity (and as a special case, periodicity) in metric and Hadamard spaces. Then, we prove a mean ergodic theorem for nonexpansive mappings and continuous semigroups of contractions in locally compact Hadamard spaces. Finally, an application to the asymptotic behavior of the first order evolution equation associated to the monotone vector field on Hadamard manifolds is presented.
\end{abstract}


\end{frontmatter}

%\linenumbers

\section{Introduction}
Von Neumann in \cite{VNeumann1932} proved the first mean ergodic theorem for linear nonexpansive mappings in Hilbert spaces. Almost forty years later
\noindent Baillon \cite{baillon1975} proved the first mean ergodic theorem for nonlinear nonexpansive mappings. Regardless of some details, Baillon's theorem is stated as follows.
\begin{theorem}\label{baillon}
Let $C$ be a nonempty closed convex subset of a Hilbert space $H$, $T:C\longrightarrow C$ be a nonexpansive mapping and $F(T):=\{x\in C: Tx=x\}\neq \emptyset$. Then for each $x\in C$, $\frac{1}{n}\displaystyle\sum_{k=0}^{n-1}T^kx$ converges weakly to a fixed point of $T$ as $n\rightarrow\infty$.
\end{theorem}
The convergence stated in Theorem \ref{baillon} for orbits of the nonexpansive mapping $T$ is called the Cesaro (or the ergodic) convergence. A stronger notion of convergence which introduced by Lorentz \cite{lorentz1948} is called {\it almost convergence}. A sequence $\{x_n\}$ is called almost convergent if $S_n^k=\frac{1}{n}\displaystyle\sum_{i=0}^{n-1}x_{k+i},$ converges as $n\to\infty$ uniformly respect to $k$.
\noindent Br\'{e}zis and Browder \cite{Brezis1976} extended Baillon's nonlinear ergodic theorem to convergence of general summability methods including almost convergence of the orbit of a nonexpansive mapping in Hilbert spaces. Reich \cite{Reich1978} improved these results and simplified their proof and in \cite{Reich1979} generalized these results to Banach spaces. Bruck \cite{Bruck1979} provided other proof for almost convergence of the orbit of a nonexpansive mapping in Banach spaces. After them, many authors and researchers studied various versions of nonlinear ergodic theorems for nonexpansive mappings and semigroups and their generalizations in linear spaces setting specially in Banach spaces.

Another concept that plays an essential role in this paper is {\it almost periodicity}. This concept was first studied by Bohr \cite{bohr1947almost} and it has a close relationship with the almost convergence. In fact, it is proved that every almost periodic sequence is almost convergent (see \cite{bohr1947almost, lorentz1948}). In this paper we first study periodicity and almost periodicity in metric and Hadamard (introduced in the next section) spaces . Then using the concept of almost periodicity, we prove an ergodic convergence theorem in locally compact Hadamard spaces.

The paper is organized as follows. Section 2 is devoted to the introduction of Hadamard spaces and some preliminaries that we need in the sequel. In Section 3, we show that periodicity and almost periodicity of a sequence is equivalent to periodicity and almost periodicity of the sequence made with the distance of the sequence points with an arbitrary point. Then, we recall a result about relation between almost periodicity and relatively compactness of orbits of a nonexpansive mapping in metric spaces. Section 3 provides the necessary tools for proving the almost convergence of orbits of a nonexpansive mapping. Finally, in Section 4, we present the main result of the paper, i.e., ergodic theorem for nonexpansive mappings in locally compact Hadamard spaces. Our method for proving the ergodic theorem is based on the strong convexity of the square of the distance in Hadamard spaces as well as a lemma on the stability of the minimum point of a strongly convex function. In Section 5, we show that the main results of Sections 3 and 4 hold for continuous semigroup of contractions. In the last section we give an application of the main result to the convergence of generated semigroup by solutions of a Cauchy problem governed by a monotone vector field on Hadamard manifols. This result extends the result of Br\'{e}zis and Baillon \cite{BaillonBerzis1976} from Hilbert spaces to Hadamard manifolds.
 
\section{Preliminaries}

 A metric space $(X,d)$ is said to be a geodesic metric space if every two points $x,y$ of $X$ are jointed by a geodesic segment (geodesic) that is the image of the isometry
\vspace{-0.2cm}
\allowdisplaybreaks\begin{equation*}
\gamma :[0,d(x,y)]\longrightarrow X,
\end{equation*}
with $\gamma(0)=x, \gamma(d(x,y))=y$ and $d\big(\gamma(t),\gamma(t')\big)=|t-t'|, \quad \forall t,t'\in[0,1]$.
Such space is said uniquely geodesic if between any two points there is exactly one geodesic that for two arbitrary points $x,y$ is denoted by $[x,y]$. All points in $[x,y]$ are denoted by $z_t=(1-t)x\oplus ty$ for all $t\in[0,1]$, where $d(z_t,x)=td(x,y)$ and $d(z_t,y)=(1-t)d(x,y)$.

A geodesic triangle $\triangle:=\triangle(x_1,x_2,x_3)$ in a geodesic space $X$ consists of three points $x_1, x_2, x_3\in X$ as vertices and three geodesic segments joining each pair of vertices as edges. A comparison triangle for the geodesic triangle $\triangle$ is the triangle $\overline{\triangle}(x_1,x_2,x_3):=\triangle(\overline{x_1},\overline{x_2},\overline{x_3})$ in the Euclidean space $\Bbb{R}^2$ such that $d(x_i,x_j)=d_{\Bbb{R}^2}(\overline{x_i},\overline{x_j})$ for all $i,j=1,2,3$. A geodesic space $X$ is said to be a $CAT(0)$ space if for each geodesic triangle $\triangle$ in $X$ and its comparison triangle $\overline{\triangle}:=\triangle(\overline{x_1},\overline{x_2},\overline{x_3})$ in $\Bbb{R}^2$, the $CAT(0)$ inequality
\begin{equation*}
d(x,y)\leq d_{\Bbb{R}^2}(\overline{x},\overline{y}),
\end{equation*}
is satisfied for all $x,y\in \triangle$ and all comparison points $\overline{x}, \overline{y}\in \overline{\triangle}$ i.e., a geodesic triangle in $X$ is at least as thin as its comparison triangle in the Euclidean plane. A $CAT(0)$ space is uniquely geodesic. A complete $CAT(0)$ space is said Hadamard space. From now, we denote every Hadamard space by $\mathscr H$.

Let $(X,d)$ be a metric space, a mapping $T:X\longrightarrow X$ is called nonexpansive if $d(Tx,Ty)\leq d(x,y), \quad \forall x,y\in X$. $F(T)=\{x\in X : Tx=x\}$ denotes the set of all fixed points of the mapping $T$, which is closed and convex in Hadamard spaces (see \cite{kirkvalencia}). A function $f:\mathscr H\longrightarrow \Bbb R$ is said to be convex if for all $x,y\in \mathscr H$ and for all $\lambda\in [0,1]$
\begin{equation*}
f\big((1-\lambda )x\oplus\lambda y\big)\leq (1-\lambda)f(x)+\lambda f(y),
\end{equation*}
also $f$ is said to be strongly convex with parameter $\gamma>0$ if for all $x,y\in \mathscr H$
\begin{equation*}
f\big(\lambda x\oplus (1-\lambda)y\big)\leq \lambda f(x)+(1-\lambda)f(y)-\lambda(1-\lambda)\gamma d^2(x,y).
\end{equation*}
A function $f:\mathscr H\longrightarrow \mathbb{R}$ is said to be lower semicontinuous (shortly, lsc) if the set $\{x\in X : f(x)\leq \alpha\}$ is closed for all $\alpha\in \Bbb R$. Any lsc, strongly convex function in a Hadamard space has a unique minimizer \cite{bacak2014convex}. The following fact can be found in \cite[Lemma 2.5]{DHOMPONGSA20082572} and \cite[page 163]{bridson2011metric}:\\
a geodesic metric space is a $CAT(0)$ space if and only if the function $d^2(x,\cdot)$, for all $x$, is strongly convex with $\gamma=1$, i.e.,
For every three points $x_0 ,x_1, y \in X$ and for every $0<t<1$
\allowdisplaybreaks\begin{equation}
d^2(y, x_t)\leq(1-t)d^2(y,x_0)+td^2(y,x_1)-t(1-t)d^2(x_0,x_1),
\end{equation}
where $x_t=(1-t)x_0\oplus tx_1$ for every $t\in[0,1]$.

Berg and Nikolaev in \cite{Berg2008} introduced the notion of quasilinearization that is the map $\langle \cdot,\cdot\rangle:(X\times X)\times (X\times X)\longrightarrow \Bbb R$ defined by 
\begin{equation}\label{e12}
\langle\overset{\rightarrow}{ab},\overset{\rightarrow}{cd}\rangle=\frac{1}{2}\big\{d^2(a,d)+d^2(b,c)-d^2(a,c)-d^2(b,d) \big\} \quad a,b,c,d\in X,
\end{equation}
where a vector $\overset{\rightarrow}{ab}$ or $ab$ denotes a pair $(a,b)\in X\times X$. In \cite{Berg2008} they proved that $CAT(0)$ spaces satisfy the Cauchy-Schwarz like inequality:
\begin{equation*}
\langle ab, cd\rangle\leq d(a,b)d(c,d) \quad (a,b,c,d\in X).
\end{equation*}

Let $\mathscr{U}$ denote the set of all relatively compact bounded sequences in a Hadamard space $\mathscr H$, i.e., the bounded sequences $\{x_n\}$ such that the set $\{x_n:\ n=1.2.3,\cdots\}$ is relatively compact. For $\epsilon>0$, we say that $E\subset K$ is an $\epsilon-$net of $K$ if for each $x\in K$ there exists $e\in E$ such that $d(x,e)<\epsilon$. A subset $K$ of a metric space $(X,d)$ is said to be totally bounded if for each $\epsilon>0$ there exist $x_1,x_2,\cdots,x_n\in X$ such that $K\subseteq \displaystyle\bigcup_{i=1}^{n} B_{\epsilon}(x_i)$ or $K$ has a finite $\epsilon-$net. It is well-known that the totally boundedness coincides with relatively compactness in complete metric spaces (see for example \cite{brown2013topological}).

%%%%%%%%%%%%%%%%%%
\section{Periodicity and Almost Periodicity}
Kurtz \cite{kurtz1970} proved that the periodicity (resp. almost periodicity) of a sequence $\{x_n\}$ in a Banach space is equivalent to the periodicity (resp. almost periodicity) of the scaler sequence  $\{x^*(x_n)\}$, where $x^*$ is an arbitrary element of the dual of the Banach space. Inspired of the results of Kurtz \cite{kurtz1970} in this section we show the equivalence between periodicity (resp. almost periodicity) of a sequence $\{x_n\}$ in a complete metric (or Hadamard) space and periodicity (resp. almost periodicity) of the real sequence $\{d(x_n,y)\}$, where $y$ is an arbitrary point of the metric space. These results correspond to the results of Kurtz \cite{kurtz1970} in a space where the dual concept is not naturally presented. Then we continue this section by studying the relation between almost periodicity and relatively compactness of orbits of a nonexpansive mapping in complete metric spaces. 
First we recall the notions of periodicity and almost periodicity in metric spaces that are natural extensions of the corresponding notions for real sequences. 
\begin{definition}
Let $\{x_n\}$ be a sequence in metric space $(X,d)$, we call this sequence is periodic with the period $p$ if there exists a positive integer $p$ such that $x_{n+p}=x_n$ for all $n$.\\
A sequence $\{x_n\}$ is called almost periodic if for each $\epsilon>0$ there are natural numbers $L=L(\epsilon)$ and $N=N(\epsilon)$ such that any interval $(k,k+L)$ where $k\geq 0$ contains at least one integer $p$ satisfying
\begin{equation}
d(x_{n+p},x_n)<\epsilon \quad\quad \forall n\geq N.
\end{equation}
\end{definition}

\begin{proposition}\label{theo-per}
A sequence $\{x_n\}$ in a Hadamard space $\mathscr{H}$ is periodic if and only if $\{d(x_n,x)\}$ is periodic for each $x\in\mathscr{H}$.
\end{proposition}
\begin{proof}
The necessity is trivial because for each $x\in\mathscr{H}$, if $p$ is the period of $\{x_n\}$, we have $d(x_{n+p},x)=d(x_n,x)$ for all $n$. Then the real sequence $\{d(x_n,x)\}$ is periodic.\\
Conversely, if for each $x\in\mathscr{H}$, the sequence $\{d(x_n,x)\}$ is periodic, by the definition for each $x\in \mathscr H$ there exists $k$ such that $d(x_{n+k},x)=d(x_n,x)$ for all $n$. Take
\begin{equation}
E_k=\{x\in\mathscr{H} : d(x_{n+k},x)=d(x_n,x), \quad\forall n\}.
\end{equation}
It is clear that by the continuity of the metric function, each $E_k$ is closed and $\mathscr{H}=\displaystyle\bigcup_{k=1}^{\infty} E_k$. By the Bair category theorem in a complete metric space, at least one of $E_k$ contains an open ball. Let $B_r(x_0)$ be the open ball i.e., there is a positive integer $k$ such that $B_r(x_0)\subset E_k$. Based on \cite[Proposition 9.2.28]{burago2001course} for all $n$, in the open ball $B_r(x_0)$, there is a geodesic segment joining $x_0$ and $x_0\neq x\in B_r(x_0)$ that is parallel with the geodesic segment joining $x_n$ and $x_{n+k}$.
\begin{figure}[!h]
\centerline{\includegraphics[height=4cm]{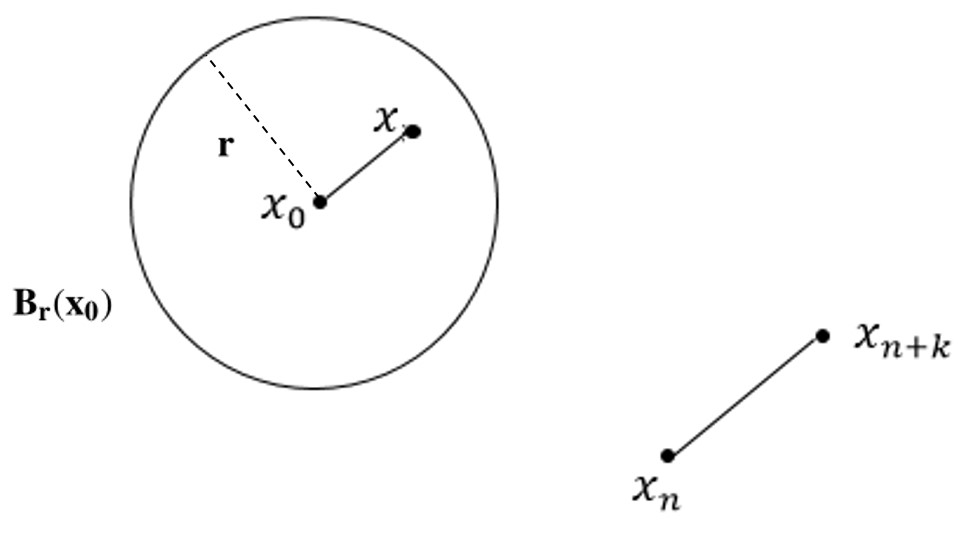}}
\end{figure}

\noindent Using quasilinearization and this fact that the Cauchy-Schwarz inequality for two parallel geodesic segments becomes equality (by \cite{Berg2007}), we have:
\begin{eqnarray}\label{e1}
&&\frac{1}{2}\bigg(d^2(x_n,x)+d^2(x_{n+k},x_0)-d^2(x_n,x_0)-d^2(x_{n+k},x) \bigg)\nonumber \\
&&=\langle x_nx_{n+k},x_0x \rangle \nonumber \\
&&= d(x_n,x_{n+k})d(x_0,x).
\end{eqnarray}
But by the definition of $E_k$ and since $x,x_0\in B_r(x_0)\subset E_k$, we get:
\begin{equation}\label{e2}
d^2(x_n,x)=d^2(x_{n+k},x) \quad \text{and} \quad d^2(x_n,x_0)=d^2(x_{n+k},x_0).
\end{equation}
Therefore \eqref{e1} and \eqref{e2} imply that:
\begin{equation*}
d(x_n,x_{n+k})d(x_0,x)=0,
\end{equation*}
since $d(x_0,x)\neq 0$, we obtain $d(x_n,x_{n+k})=0$ for all $n$ i.e., $x_n=x_{n+k}$ for all $n$. Thus, the sequence $\{x_n\}$ is periodic.
\end{proof}

\begin{proposition}\label{theo-almper}
A sequence $\{x_n\}$ in $\mathscr{U}$ is almost periodic if and only if $\{d(x_n,x)\}$ is almost periodic for each $x\in\mathscr{H}$.
\end{proposition}
\begin{proof}
{\bf Necessity.}
Since $\{x_n\}$ is almost periodic, for each $\epsilon >0$, there exist natural numbers $L=L(\epsilon), N=N(\epsilon)$ such that every interval $(k,k+L)$ for any $k\geq 0$, contains at least one integer $p$ such that $d(x_{n+p},x_n)<\epsilon$ holds for all $n\geq N$. Also for each $x\in\mathscr{H}$, we have:
\begin{equation*}
|d(x_{n+p},x)-d(x_n,x)|\leq d(x_{n+p},x_n).
\end{equation*}
So the same $N, L$ and $p$ in the definition of almost periodicity of $\{x_n\}$ show that for each $x$, $\{d(x_n,x)\}$ is almost periodic.\\
{\bf Sufficiency.} Let $\epsilon>0$. There are $q_1,q_2,\cdots,q_s\in X$ such that the union of the family $\{B_{\frac{\epsilon}{3}}(q_i)\}_{i=1}^{s}$ contains $\{x_n\}$. Since for each $x$, $\{d(x_n,x)\}$ is almost periodic, for each $q_i$ where $1\leq i\leq s$, $\{d(x_n,q_i)\}$ is almost periodic. Therefore for all $1\leq i\leq s$, there are $L_i, N_i$ such that any interval $(k,k+L_i)$, where $k\geq 0$ contains $p$ satisfying
\begin{equation}\label{e3}
|d(x_{n+p},q_i)-d(x_n,q_i)|< \frac{\epsilon}{3}\quad \forall n\geq N_i.
\end{equation}
Take $L=\displaystyle\max_{1\leq i\leq s} L_i$ and $N=\displaystyle\max_{1\leq i\leq s} N_i$. For all $n\geq N$, there is $q_{i_n}$ such that\linebreak $x_n\in B_{\frac{\epsilon}{3}}(q_{i_n})$ i.e.,
\begin{equation}\label{e4}
d(x_n,q_{i_n})< \frac{\epsilon}{3}.
\end{equation}
By \eqref{e3} for any $k\geq 0$ there is $p\in (k,k+L)$ such that
\begin{equation}\label{e5}
d(x_{n+p},q_{i_n})< \frac{2\epsilon}{3}\quad \forall n\geq N.
\end{equation}
Thus \eqref{e4} and \eqref{e5} imply that for arbitrary $\epsilon >0$ there are $L, N$ such that for any $k\geq 0$ there is at least one $p$ in the interval $(k,k+L)$ satisfying 
\begin{equation*}
d(x_{n+p},x_n)< \epsilon \quad \forall n\geq N,
\end{equation*}
and the proof of almost periodicity of $\{x_n\}$ is complete.
\end{proof}
The following lemma shows that Theorem 1 of \cite{DAFERMOS197397} is also true for general metric spaces. We use this lemma to extend the next theorem from Banach spaces to complete metric spaces. It is easily seen that the proof of \cite{DAFERMOS197397} also works for metric spaces, then we recall the lemma without proof.   
\begin{lemma}\label{l1}
Let $(X,d)$ be a metric space and $T:X\to X$ be a nonexpansive mapping. If $w(x)$ is the set of strong subsequential limits of the iteration sequence $\{T^nx\}$, then $T$ is an isometry on $w(x)$.
\end{lemma}
The next theorem extends \cite[Theorem 3.2]{Baillonbruckreich} from Banach spaces to complete metric spaces. Although the proof of the theorem is briefly summarized in \cite{Baillonbruckreich}, we facilitate the reader with the following more complete proof.  
\begin{theorem}\label{theo-almperreco}
Suppose that $(X,d)$ is a complete metric space and $T:X\rightarrow X$ is a nonexpansive mapping with $F(T)\neq\emptyset$. Then the sequence $\{T^nx\}$ is almost periodic if and only if it is relatively compact.
\end{theorem}
\begin{proof}
Necessity is obvious because by \cite{vesel2011} an almost periodic sequence is totally bounded, and we know in a complete metric space totally boundedness is equivalent to relatively compactness. For converse by assumption, $\{T^nx\}$ is relatively compact. If we denote the set of all strongly subsequential limits of $\{T^nx\}$ with $w(x)$, then this set is nonempty, closed and consequently compact by relatively compactness of the iteration sequence. Also, by Lemma \ref{l1} $T$ is isometry on $w(x)$. Take an element $y_0$ in $w(x)$ and set $y_n=T^ny_0$. Clearly since $T$ is isometry, $\{y_n\}$ is a relatively compact isometric sequence i.e., for all $n, m, i$ in $\mathbb{Z}^+$ we have $d(y_{n+i},y_{m+i})=d(y_n,y_m)$. First we show that $\{y_n\}$ is almost periodic sequence.\\
By relatively compactness, let $\{y_{m_i}\}_{i=1}^{r}$ be a finite $\epsilon-$net for $\{y_n\}$.
\begin{figure}[!h]
\centerline{\includegraphics[height=3cm]{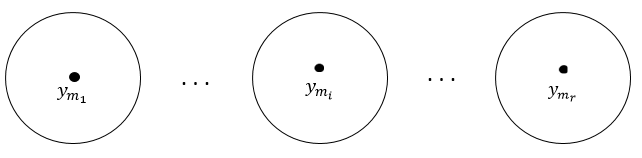}}
\end{figure}
Set $L=r+1$ and $N=0$ in the definition of almost periodicity because for arbitrary $n\geq 0$, $y_n$ belongs to $B_{\epsilon}(y_{m_i})$ where $1\leq i\leq r$. For $y_n, y_{n+1},\cdots, y_{n+r-1}$ if at least two points belong to one ball, in fact $y_{n+i}$ and $y_{n+i+p}$ where $p<L=r+1$ are in the same ball and by isometricity of the sequence $\{y_n\}$ we have:
\begin{equation}\label{e6}
d(y_{n+i},y_{n+i+p})=d(y_n,y_{n+p})<\epsilon.
\end{equation}
If none of $y_n, y_{n+1},\cdots, y_{n+r-1}$ belong to one ball, certainly $y_{n+r}$ with one of them place exactly into a ball, so $y_{n+j}$ and $y_{n+j+p}$ where $p<L=r+1$ are into the same ball and again by isometricity of the sequence $\{y_n\}$, similar to \eqref{e6} $d(y_n,y_{n+p})<\epsilon$. Thus, we find one $p$ in the interval $(0,r+1)$ such that by isometricity of the sequence $\{y_n\}$ for all $n\geq 0$, $d(y_n,y_{n+p})<\epsilon$. Similar arguments apply to the case arbitrary $k$ for the existence of $p$ in the interval $(k,k+L)$ and this implies that $\{y_n\}$ is almost periodic.\\
Now since $y_0$ is limit point of $\{T^nx\}$, there exists $N_1=N(\frac{\epsilon}{3})$ such that for all $n\geq N_1$:
\begin{equation}\label{e7}
d(T^nx,y_0)<\frac{\epsilon}{3}.
\end{equation}
Also we see that $\{y_n\}$ is almost periodic, thus by necessity for $\frac{\epsilon}{3}$ there is a finite $\frac{\epsilon}{3}-$net with $r$ ball that cover the sequence $\{y_n\}$, so if $L=r+1$ and $N_2=0$, for any $k\geq 0$ there is at least one $p$ in the interval $(k,k+L)$ such that for all $n\geq 0$, $d(y_{n+p},y_n)<\frac{\epsilon}{3}$ and therefore: 
\begin{equation}\label{e8}
d(T^py_0,y_0)<\frac{\epsilon}{3}.
\end{equation}
Finally for arbitrary $\epsilon>0$, if we set $N=N_1$ and $L=r+1$ where $r$ is the number of balls in $\frac{\epsilon}{3}-$net, since $T$ is nonexpansive and by \eqref{e7} and \eqref{e8} if for all $k\geq 0$ in the interval $(k,k+L)$ we take the same $p$ in \eqref{e8}, for all $n\geq N$ we obtain:
\begin{eqnarray}
d(T^{n+p}x,T^nx)&\leq & d(T^{N+p}x,T^Nx) \nonumber \\
&\leq &d(T^{N+p}x,T^py_0)+d(T^py_0,y_0)+d(y_0,T^Nx)\nonumber\\
&\leq &d(T^Nx,y_0)+d(T^py_0,y_0)+d(y_0,T^Nx)\nonumber\\
&\leq &\frac{\epsilon}{3}+\frac{\epsilon}{3}+\frac{\epsilon}{3}=\epsilon, \nonumber
\end{eqnarray}
and this completes the proof.
\end{proof}

An immediate consequence of Theorem \ref{theo-almperreco} is the almost periodicity of orbits of a nonexpansive mapping with nonempty fixed points set in a locally compact metric space. We will use this property in the next section to prove the mean ergodic theorem.

%%%%%%%%%%%%%%%%%%%%%%%%%
\section{Ergodic Theorem for Nonexpansive Mappings}
In Hadamard spaces, Liimatainen \cite{Liimatainen2012} proved a result related to the mean ergodic convergence for the Karcher mean of orbits of nonexpansive mappings. The best result of \cite{Liimatainen2012} for ergodic convergence of orbits of a general nonexpansive mapping implies that every weak cluster point (for the definitions of weak cluster point and weak convergence in Hadamard spaces see \cite{bacak2014convex, Kirk2008weak}) of the Karcher mean (see the following definition) of a bounded orbit is a fixed point of the mapping. But the weak convergence of the orbit is still an open problem in general Hadamard spaces, because the set of weak cluster points is not necessarily a singleton. In this paper using the notion of almost periodicity studied in Section 2 in Hadamard spaces (Proposition \ref{theo-almper} and Theorem \ref{theo-almperreco}), we prove the almost convergence of the orbit (which is stronger than the ergodic convergence) in locally compact Hadamard spaces including Hadamard (finite dimensional) manifolds. First we recall the Karcher mean concept in Hadamard spaces.
\begin{definition}\label{Karcher mean}
Give a sequence $\{x_n\}$ in a Hadamard space and $n\in \mathbb{N}$ and $k\in \mathbb{N}\cup\{0\}$, we define the functions
\begin{equation}\label{e22}
\mathcal{F}_n(x)=\frac{1}{n}\displaystyle\sum_{i=0}^{n-1} d^2(x_i,x),
\end{equation}
and
\begin{equation}\label{e23}
\mathcal{F}_n^k(x)=\frac{1}{n}\displaystyle\sum_{i=0}^{n-1} d^2(x_{k+i},x).
\end{equation}
From \cite[Proposition 2.2.17]{bacak2014convex} we know that these functions have unique minimizers. For $\mathcal{F}_n(x)$ the unique minimizer is denoted by $\sigma_n(x_0,\ldots ,x_{n-1})$ (shortly, $\sigma_n$) and is called the mean of $x_0,\ldots ,x_{n-1}$. Also for $\mathcal{F}_n^k(x)$ the unique minimizer is denoted by $\sigma_n^k(x_k,\ldots ,x_{k+n-1})$ (shortly, $\sigma_n^k$) and is called the mean of $x_k,\ldots ,x_{k+n-1}$. These means are known as the Karcher means \cite{karcher1977} and in Hilbert spaces they coincide with the linear (usual) means (see \cite[example 2.2.5]{bacak2014convex}).
 A sequence $\{x_n\}$ in a Hadamard space $\mathscr H$ is called the Cesaro convergent or the mean convergent  (resp. almost convergent) to $x\in X$, if $\sigma_n$ (resp. $\sigma_n^k$) converges (resp. converges uniformly in $k$) to $x$. 
\end{definition}
\noindent For an orbit $\{T^n x | n=0,1,2,\ldots\}$, $\sigma_n(x)$ and $\sigma_n^k(x)$, are defined respectively as the unique minimizers of the functions 
\begin{equation*}
\mathcal{F}[x]_n(y)=\frac{1}{n}\displaystyle\sum_{i=0}^{n-1} d^2(T^ix,y),
\end{equation*}
and
\begin{equation*}
\mathcal{F}[x]_n^k(y)=\frac{1}{n}\displaystyle\sum_{i=0}^{n-1} d^2(T^{k+i}x,y).
\end{equation*}

 Two following lemmas are needed to state the main result.
\begin{lemma}\label{lemstconfun1}
Let $(\mathscr H, d)$ be a Hadamard space, $f:\mathscr H\longrightarrow \Bbb R$ be a lower semicontinuous and strongly convex function and $x$ be the unique minimizer of $f$. Then for every $y$ out of a neighborhood centered at $x$ with radius $\delta$, we have:
\begin{equation*}
f(x)<f(y) - \big(d(x,y)-\delta\big)d(x,y).
\end{equation*}
\end{lemma}
\begin{proof}
Let $y$ be a point out of a neighborhood centered at $x$ with radius $\delta$. The geodesic segment joining $x$ and $y$ intersects this ball at $\lambda_0 x+(1-\lambda_0)y$, where $0<\lambda_0<1$.
\begin{figure}[!h]
\centerline{\includegraphics[height=3cm]{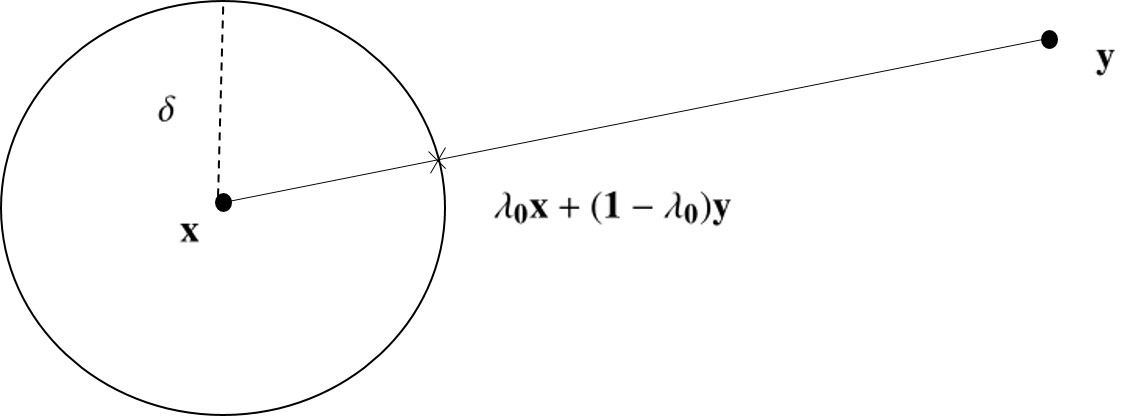}}
\end{figure}
\noindent By properties of the metric function in Hadamard space, we have $(1-\lambda_0)d(x,y)=\delta$ so $d(x,y)=\delta +\lambda_0 d(x,y)$
and hence
\begin{equation}\label{ineqstconeq1}
\lambda_0=\frac{d(x,y)-\delta}{d(x,y)}.
\end{equation}
Since $x$ is the unique minimizer of the strongly convex function $f$, we have:
\allowdisplaybreaks\begin{eqnarray}
f(x)& <& f\big(\lambda_0 x+(1-\lambda_0)y\big)\nonumber \\
&\leq &\lambda_0 f(x)+(1-\lambda_0)f(y)- \lambda_0(1-\lambda_0) d^2(x,y), \nonumber
\end{eqnarray}
therefore
\begin{equation*}
(1-\lambda_0)f(x)<(1-\lambda_0)f(y) - \lambda_0(1-\lambda_0)d^2(x,y).
\end{equation*}
Dividing by $(1-\lambda_0)$, we get the desired result by \eqref{ineqstconeq1}.
\end{proof}
\begin{lemma}\label{lemstconfun2}
Let $(\mathscr H, d)$ be a Hadamard space and $\{f_n^k\}_{k,n}$ be the sequence of convex functions on $\mathscr H$. If $\{x_n^k\}_{k,n}$ is a sequence of minimum points of $\{f_n^k\}_{k,n}$ and $x$ is the unique minimizer of the strongly convex function $f$, satisfying:
\begin{enumerate}
\renewcommand{\theenumi}{\Roman{enumi}}
\item\label{lemstconfun2i1} the sequence $\{f_n^k\}$ is pointwise convergent to $f$ as $n$ tends to infinity uniformly in $k\geq 0$,
\item\label{lemstconfun2i2} $\displaystyle\limsup_{n\rightarrow\infty} \sup_{k\geq 0}\big(f(x_n^k)-f_n^k(x_n^k)\big)\leq 0$.
\end{enumerate}
Then $x_n^k$ converges to $x$ uniformly in $k\geq 0$ as $n\rightarrow \infty$.
\end{lemma}
\begin{proof}
The goal is to prove that $\displaystyle\lim_{n\rightarrow\infty}\sup_{k\geq 0} d(x_n^k,x)=0$. Suppose to the contrary,
\begin{equation*}
\exists \delta>0 \quad :\quad \forall N\quad \exists n\geq N\quad \text{such that}\quad  \sup_{k\geq 0} d(x_n^k,x)\geq\delta.
\end{equation*}
Therefore for each $0<\epsilon<\delta$ there exist $k=k(\epsilon,n)$, and a subsequence of $\{x_n^k\}_n$ (we denote it by the same sequence $\{x_n^k\}_n$) such that $d(x_n^k,x)\geq\delta-\epsilon$.
If we choose $\epsilon<\frac{\delta}{2}$, we have:
\begin{equation}\label{ineqstconeq2}
d(x_n^k,x)\geq\delta-\epsilon>\frac{\delta}{2}.
\end{equation}
Using Lemma \ref{lemstconfun1} for $\frac{\delta}{2}$ instead of $\delta$, we get:
\begin{equation}\label{ineqstconeq3}
f(x)<f(x_n^k) - \big(d(x,x_n^k)-\frac{\delta}{2}\big)d(x,x_n^k).
\end{equation}
On the other hand since $x_n^k$ is the minimum point of $f_n^k$, we have:
\begin{equation}\label{ineqstconeq4}
f_n^k(x_n^k)\leq f_n^k(x).
\end{equation}
By \eqref{ineqstconeq2}, \eqref{ineqstconeq3} and \eqref{ineqstconeq4}, we obtain:
\allowdisplaybreaks\begin{eqnarray}
(\frac{\delta}{2}-\epsilon)\frac{\delta}{2}&< & \big(d(x,x_n^k)-\frac{\delta}{2}\big)d(x,x_n^k)\nonumber \\
&< &f(x_n^k)-f(x)+f_n^k(x)-f_n^k(x_n^k). \nonumber
\end{eqnarray}
By the assumptions \eqref{lemstconfun2i1} and \eqref{lemstconfun2i2}, we get a contradiction. Thus $x_n^k$ converges to $x$ uniformly in $k\geq 0$ as $n\rightarrow \infty$.
\end{proof}
\begin{theorem}\label{maintheo1}
Let $C$ be a nonempty, closed and convex subset of a locally compact Hadamard space $\mathscr{H}$ and $T:C\rightarrow C$ be a nonexpansive mapping with $F(T)\neq\emptyset$. Then the sequence $\{T^nx\}$ is almost convergent to a fixed point of $T$.
\end{theorem}
\begin{proof}
Let $x\in \mathscr H$. Since $F(T)\neq\emptyset$, $\{T^nx\}$ is bounded and by the local compactness of the space, it is relatively compact. Theorem \ref{theo-almperreco} implies that $\{T^nx\}$ is almost periodic and by Proposition \ref{theo-almper} $\{d^2(T^nx,y)\}$ is almost periodic for all $y\in \mathscr H$. By \cite{lorentz1948}(see also \cite{bohr1947almost}) the scalar sequence $\{d^2(T^nx,y)\}$ is almost convergent for all $y\in \mathscr H$. Define:
\begin{equation}\label{eqmainresu1}
\mathcal{F}[x]_n^k(y):=\displaystyle\frac{1}{n}\sum_{i=0}^{n-1}d^2(T^{k+i}x,y),
\end{equation}
and
\begin{equation}\label{eqmainresu2}
\mathcal{F}[x](y):=\displaystyle\lim_{n\rightarrow\infty}\frac{1}{n}\sum_{i=0}^{n-1}d^2(T^{k+i}x,y)\quad \text{uniformly in}\ k\geq 0.
\end{equation}
Almost convergency of $\{d^2(T^nx,y)\}$ for any $y\in \mathscr H$ shows that \eqref{eqmainresu2} is well defined. By the strong convexity of $d^2(\cdot,x)$, the functions $\mathcal{F}[x]_n^k$ and $\mathcal{F}[x]$ are strongly convex and therefore have unique minimizers $\sigma_n^k(x)$ and $\sigma(x)$ respectively. If we show that assumption \eqref{lemstconfun2i2} of Lemma \ref{lemstconfun2} holds, then $\mathcal{F}[x]_n^k$ and $\mathcal{F}[x]$ satisfy all assumptions of Lemma \ref{lemstconfun2} and hence $\{T^nx\}$ is almost convergent to $\sigma(x)$. But by the first part of Lemma \ref{lemstconfun2} there exists a subsequence $\{\sigma_n^k(x)\}_n$ satisfying inequalities this lemma. We have to show
\begin{equation*}
\displaystyle\limsup_{n\rightarrow\infty} \sup_{k\geq 0}\Big(\mathcal{F}[x]\big(\sigma_n^k(x)\big)-\mathcal{F}[x]_n^k\big(\sigma_n^k(x)\big)\Big)\leq 0.
\end{equation*}
Suppose to the contrary
\begin{equation*}
\displaystyle\limsup_{n\rightarrow\infty} \sup_{k\geq 0}\Big(\mathcal{F}[x]\big(\sigma_n^k(x)\big)-\mathcal{F}[x]_n^k\big(\sigma_n^k(x)\big)\Big)> 0,
\end{equation*}
i.e.,
\begin{equation*}
\exists \lambda>0\quad :\quad \displaystyle\limsup_{n\rightarrow\infty} \sup_{k\geq 0}\Big(\mathcal{F}[x]\big(\sigma_n^k(x)\big)-\mathcal{F}[x]_n^k\big(\sigma_n^k(x)\big)\Big)> \lambda.
\end{equation*}
Therefore there exists a subsequence $\{n_i\}$ of $\{n\}$ such that:
\begin{equation*}
\sup_{k\geq 0}\Big(\mathcal{F}[x]\big(\sigma_{n_i}^k(x)\big)-\mathcal{F}[x]_{n_i}^k\big(\sigma_{n_i}^k(x)\big)\Big)> \lambda.
\end{equation*}
On the other hand for each $0<\epsilon<\lambda$, there exists $k=k(\epsilon)$ such that $$\mathcal{F}[x]\big(\sigma_{n_i}^k(x)\big)-\mathcal{F}[x]_{n_i}^k\big(\sigma_{n_i}^k(x)\big)>\lambda -\epsilon,$$ i.e.,
\begin{equation}\label{eqmainresu3}
\mathcal{F}[x]\big(\sigma_{n_i}^k(x)\big)>\lambda -\epsilon +\mathcal{F}[x]_{n_i}^k\big(\sigma_{n_i}^k(x)\big).
\end{equation}
By \eqref{eqmainresu2} and \eqref{eqmainresu3}, there exists $p_0$ such that for all $p\geq p_0$ we have:
\begin{equation*}
\displaystyle\frac{1}{p}\sum_{i=0}^{p-1} d^2\big(T^{k+i}x,\sigma_{n_i}^k(x)\big)>\lambda-\epsilon+\displaystyle\frac{1}{n_i}\sum_{i=0}^{n_i-1} d^2\big(T^{k+i}x,\sigma_{n_i}^k(x)\big).
\end{equation*}
Taking $p=n_i$ for sufficiently large $i$ such that $n_i\geq p_0$, we obtain: 
\begin{equation*}
\displaystyle\frac{1}{n_i}\sum_{i=0}^{n_i-1} d^2\big(T^{k+i}x,\sigma_{n_i}^k(x)\big)>\lambda-\epsilon+\displaystyle\frac{1}{n_i}\sum_{i=0}^{n_i-1} d^2\big(T^{k+i}x,\sigma_{n_i}^k(x)\big),
\end{equation*}
which is a contradiction. This shows that assumption \eqref{lemstconfun2i2} of Lemma \ref{lemstconfun2} holds. Now we show that $\sigma(x)$ is a fixed point of $T$. We have:
\allowdisplaybreaks\begin{eqnarray}
\mathcal{F}[x]\big(T\sigma(x)\big)&=&\displaystyle\lim_{n\rightarrow\infty}\frac{1}{n}\sum_{i=0}^{n-1} d^2\big(T^{k+i}x,T\sigma(x)\big)\nonumber\\
&= &\displaystyle\lim_{n\rightarrow\infty}\frac{1}{n}\big\{d^2\big(T^kx,T\sigma(x)\big)+\displaystyle\sum_{i=1}^{n-1} d^2\big(T^{k+i}x,T\sigma(x)\big)\big\}\nonumber\\
&\leq &\displaystyle\lim_{n\rightarrow\infty}\frac{1}{n}\sum_{i=0}^{n-2} d^2\big(T^{k+i}x,\sigma(x)\big)\ \ \ \ \ \text{by the boundedness of iterations}\nonumber\\
&\leq &\displaystyle\lim_{n\rightarrow\infty}\frac{1}{n}\sum_{i=0}^{n-1} d^2\big(T^{k+i}x,\sigma(x)\big)\nonumber\\
&=&\mathcal{F}[x]\big(\sigma(x)\big).\nonumber
\end{eqnarray}
Uniqueness of the minimizer implies that $T\sigma(x)=\sigma(x)$, which completes the proof.
\end{proof}
%%%%%%%%%%%%%%%%%%%%%%%%%%%%%%%%%%%
\section{Ergodic Theorem for Continuous Semigroup of Contractions}
In this section, we state the analogous results of the two previous sections for continuous semigroup of contractions with a nonempty fixed point set and show the almost convergence of the orbit to a fixed point of the semigroup in a locally compact Hadamard space. Since the proofs are similar to the dicrete version we will state only the results without proofs. 

Let $C$ be a closed convex subset of a Hadamard space $(\mathscr{H},d)$. A one-parameter continuous semigroup of contractions is the set $\mathcal{S}=\{S(t) : 0\leq t<\infty\}$ of self-mappings $S(t):C\to C$ that satisfying the following conditions:
\begin{enumerate}
\renewcommand{\theenumi}{\roman{enumi}}
\item for all $x\in X$, $S(0)x=x$,
\item for all $x\in X$ and all $t,s\in [0,\infty)$, $S(t+s)x=S(t)S(s)x$,
\item for each $x\in X$, $S(t)x$ is continuous for $t\in [0,\infty)$,
\item for each $t\in [0,\infty)$, $S(t)$ is a nonexpansive mapping on $C$,
\end{enumerate}
Let $F(\mathcal{S})$ denotes the common fixed point set of the family $\mathcal{S}$, i.e., $F(\mathcal{S})=\displaystyle\bigcap_{t\geq 0} F\big(S(t)\big)$. Note that $F(\mathcal{S})$ by \cite{kirkvalencia} is a closed and convex set in a Hadamard space.
\begin{definition}[Almost periodic function]
Let $(X,d)$ be a general metric space and $f:\Bbb R^+\longrightarrow X$ be a continuous function. $f$ is called almost periodic function if for each $\epsilon>0$, there exist real numbers $L=L(\epsilon)$ and $\tau=\tau(\epsilon)$ such that any interval $(q, q+L)$ where $q\geq 0$ contains at least one number $T$ for which 
\begin{equation*}
d\big(f(t+T),f(t)\big)<\epsilon \quad \forall t\geq \tau.
\end{equation*}
\end{definition}
Considering definition of almost periodic function and replacing sequence with net, all the argument in Proposition \ref{theo-almper} remain valid, therefore if $\mathscr{U}$ denotes all relatively compact nets in Hadamard spaces, we have the next proposition.
\begin{proposition}
A net $\{x_t\}_{t\in \Bbb R^+}$ is almost periodic if and only if $\{d(x_t,x)\}_{t\in \Bbb R^+}$ is almost periodic for each $x\in\mathscr{H}$.
\end{proposition}
 The proof of \cite[Theorem 1]{DAFERMOS197397} is stated for continuous semigroups in Banach space and also works for metric spaces. Therefore the continuous analogous of Lemma \ref{l1} holds.
\begin{lemma}
Let $C$ be a closed convex subset of a metric space $(X,d)$, $\mathcal{S}$ be a continuous semigroup of nonexpansive mappings on $C$, $x\in X$, and $w(x)$ be the set of strong subsequential of limits of the orbit $\{S(t)x\}$. Then for each $t\in [0,\infty)$, $S(t)$ is isometry on $w(x)$.
\end{lemma}
\noindent Thus Theorem \ref{theo-almperreco} remain valid for semigroup of contractions with similar proof, that we can state it as follows.
\begin{theorem}
Suppose that $(X,d)$ is a complete metric space, $C$ is a closed convex subset of $X$ and $\mathcal{S}$ is a continuous semigroup of nonexpansive mappings on $C$ with $F(\mathcal{S})\neq \emptyset$. Then for $x\in X$ the function $t\rightarrow S(t)x$ is almost periodic if and only if it's range is relatively compact.
\end{theorem}
For an orbit $\{S(t)x\}$, $\sigma_T(x)$ and $\sigma_T^s(x)$, are defined respectively as the unique minimizers of the functions 
\begin{equation*}
\mathcal{G}[x]_T(y)=\frac{1}{T}\displaystyle\int_{0}^{T} d^2(S(t)x,y)dt,
\end{equation*}
and
\begin{equation*}
\mathcal{G}[x]_T^s(y)=\frac{1}{T}\displaystyle\int_{0}^{T} d^2(S(s+t)x,y)dt.
\end{equation*}
 A net $\{S(t)x\}$ in a Hadamard space $X$ is called the Cesaro convergent or the mean convergent  (resp. almost convergent) to $x\in X$, if $\sigma_T$ (resp. $\sigma_T^s$) converges (resp. converges uniformly in $s$) to $x$.\\
Using net instead of sequence in Lemmas \ref{lemstconfun1} and \ref{lemstconfun2} implies the same argument remains true, and we have:
\begin{lemma}
Let $(\mathscr H, d)$ be a Hadamard space and $\{f_t^s\}_{s,t}$ be the net of convex functions on $\mathscr H$. If $\{x_t^s\}_{s,t}$ is a net of minimum points of $\{f_t^s\}_{s,t}$ and $x$ is the unique minimizer of the strongly convex function $f$, satisfying:
\begin{enumerate}
\renewcommand{\theenumi}{\Roman{enumi}}
\item\label{lemstconfun2i1} the net $\{f_t^s\}$ is pointwise convergent to $f$ as $t\rightarrow\infty$ uniformly in $s\geq 0$,
\item\label{lemstconfun2i2} $\displaystyle\limsup_{t\rightarrow\infty} \sup_{s\geq 0}\big(f(x_t^s)-f_t^s(x_t^s)\big)\leq 0$.
\end{enumerate}
Then $x_t^s$ converges to $x$ uniformly in $s\geq 0$ as $t\rightarrow \infty$.
\end{lemma}
Therefore the main result holds for semigroup of contractions.
\begin{theorem}\label{semimaintheo1}
Let $C$ be a nonempty, closed and convex subset of a locally compact Hadamard space $\mathscr{H}$ and $\mathcal{S}$ be a continuous semigroup of nonexpansive mappings on $C$ with $F(\mathcal{S})\neq \emptyset$. Then the orbit $\{S(t)x\}$ is almost convergent to a common fixed point of $\mathcal{S}$.
\end{theorem}
\begin{proof}
The used techniques in Theorem \ref{maintheo1} is also applicable for nets, hence for continuous semigroup of nonexpansive mappings we have $\sigma_T^s(x)\rightarrow \sigma(x)$ uniformly in $s\geq 0$ as $T\rightarrow\infty$. Now we show that $\sigma(x)$ is a common fixed point of $\mathcal{S}$. For each $r\geq 0$, without loss of generality, taking $s=u+r\geq r$, we have:
\allowdisplaybreaks\begin{eqnarray}
\mathcal{G}[x]\big(S(r)\sigma(x)\big)&=&\displaystyle\lim_{T\rightarrow\infty}\frac{1}{T}\int_{0}^{T} d^2\big(S(s+t)x,S(r)\sigma(x)\big) dt\nonumber\\
&\leq &\displaystyle\lim_{T\rightarrow\infty}\frac{1}{T}\int_{0}^{T} d^2\big(S(u+t)x,\sigma(x)\big) dt\nonumber\\
&=&\mathcal{G}[x]\big(\sigma(x)\big).\nonumber
\end{eqnarray}
Uniqueness of the minimizer implies that $S(r)\sigma(x)=\sigma(x)$ For each $r\geq 0$, which completes the proof.
\end{proof}
%%%%%%%%%%%%%%%%%%%%%%%%%%%%%%%%%%%%%%%
\section{Ergodic Convergence of Semigroups Generated by Monotone Vector Fields }
Consider a Hilbert space $H$ with the inner product $\langle\cdot ,\cdot \rangle$ and the inductive norm $\|\cdot\|$, a set valued mapping $A: H\rightarrow 2^H$ is said to be monotone if for all $x,y\in D(A)$
\begin{equation*}
\langle x^*-y^*,x-y\rangle \geq 0 \quad \forall~x^*\in Ax,~y^*\in Ay,
\end{equation*}
where $D(A)=\{x\in H : Ax\neq \emptyset\}$. The monotone operator $A$ is said to be maximal monotone if there exists no monotone operator such that it's graph properly contains $Gr(A)$, where $Gr(A)=\{(x,y)\in H\times H : x\in D(A), y\in Ax \}$. The set of zeros of the operator $A$ is $A^{-1}0=\{x\in H : Ax=0\}$. For more details about monotone operators in Hilbert spaces the reader can see \cite{bauschke2017convex}.

Let $A$ be a maximal monotone operator on $H$, the Cauchy problem
\begin{equation}\label{eveq1}
\begin{cases}
-x^{\prime}(t)\in Ax(t), &\\
 x(0)=x_0, & 
\end{cases}
\end{equation}
was studied by many authors and researchers for existence of solutions and asymptotic behavior. Taking $S(t)x_0:=x(t)$, by existence and uniqueness of solution of \eqref{eveq1}, $\mathcal{S}=\{S(t) : t\geq 0\}$ is a continuous semigroup. It is easily seen that the monotonicity of $A$ implies that $\mathcal{S}=\{S(t) : t\geq 0\}$ is a semigroup of contractions (see \cite{morosanu1988nonlinear}). It is known that $p\in A^{-1}0$ if and only if $p$ is common fixed point of the semigroup $\mathcal{S}$, therefore, $A^{-1}0$ is the same common fixed points of $S(t),\ \ t\geq0$ (see \cite{morosanu1988nonlinear}). Therefore if $A^{-1}0\neq \emptyset$, ergodic theorem proved by Baillon and Br\'{e}zis in \cite{BaillonBerzis1976}, shows that for any $y\in H$, the orbit $\{S(t)y\}$ is mean convergent to a point in the set of zeros of $A$ or common fixed point of the semigroup $\mathcal{S}$. By the main theorem of the previous section we extend the result to Hadamard manifolds.  

Hadamard manifold is a complete, simply connected $n$-dimensional Reimannian manifold of nonpositive sectional curvature. Hadamard manifolds are examples of Hadamard spaces \cite{bacak2014convex}, which are locally compact. Let $M$ be a Hadamard manifold, for $p\in M$, $T_pM$ and $TM=\displaystyle\bigcup_{p\in M} T_pM$ are denoted the tangent space at $p$ and tangent boundle of $M$ respectively. The exponential map $exp_p: T_pM\rightarrow M$ at $p$ is defined by $exp_p(v)=\gamma_v(1)$ for each $v\in T_pM$, where $\gamma_v(\cdot)$ is the geodesic with $\gamma_v(0)=p$ and $\dot{\gamma}_v(0)=v$. A mapping $A: D(A)\subset M\rightarrow 2^{TM}$ that each $x\in M$ maps to a subset of $T_xM$, is a monotone multivalued vector field iff
\begin{equation*}
\langle x^*,exp_x^{-1}y \rangle +\langle y^*,exp_y^{-1}x \rangle \leq 0,
\end{equation*}
for all $x,y\in D(A)$ and $x^*\in Ax, y^*\in Ay$. For a complete bibliography and more details about the basic concepts of Hadamard manifolds, the exponential map and also monotone vector fields, we refer the reader to \cite{sakai1996riemannian, nemeth1999, lopezmarquez2009}. If the mapping $A: D(A)\subset M\rightarrow 2^{TM}$ be a monotone multivalued vector field, by \cite[Theorems 5.1 and 5.2]{Iwamiya2003}, \eqref{eveq1} has a global solution that by Proposition 4.2 of \cite{ahmadikhatib2018} is unique. Let $S(t)x_0=x(t)$, then $\mathcal{S}=\{S(t) : t\geq 0\}$ is a nonexpansive semigroup (see lemma 4.1 of \cite{ahmadikhatib2018}). It is easy to see that the set of singularities of $A$ (i.e. the set $A^{-1}(0)$) is equal to the set of common fixed points of $\mathcal{S}$. Now we have the following result as an application of Theorem \ref{semimaintheo1}.

\begin{theorem} Suppose $A: D(A)\subset M\rightarrow 2^{TM}$ is a monotone vector field with at least a singularity point. Then every orbit $\{S(t)x\}$ of the semigroup generated by solutions of \eqref{eveq1} is almost convergent to a singularity of $A$. 
\end{theorem} 

%%%%%%%%%%%%%%%%%%%%%%%%%%%%%%%%%%%%

%\section*{References}

\bibliography{mybibfile}

\end{document}